\documentclass[a4paper,11pt]{amsart}

\usepackage{amsmath}
\usepackage{amsthm}
\usepackage[all]{xy}
\usepackage{amssymb}
\usepackage{mathrsfs}
\usepackage{enumerate}
\usepackage[usenames,dvipsnames]{color}

\usepackage{fullpage}
\usepackage{amsmath,amsfonts,amssymb}
\makeatletter
\newcommand{\thorn}{{\fontencoding{T1}\selectfont\th}}
\newcommand{\@indepsymbol}[2]{#1\setbox0=\hbox{$#1x$}\kern\wd0\hbox to 0pt{\hss$#1\mid$\hss}\lower.9\ht0\hbox to 0pt{\hss$#1\smile$\hss}\kern\wd0}
\newcommand{\@nindepsymbol}[2]{#1\setbox0=\hbox{$#1x$}\kern\wd0\hbox to 0pt{\mathchardef
	\nn=12854\hss$#1\nn$\kern1.4\wd0\hss}\hbox to
	0pt{\hss$#1\mid$\hss}\lower.9\ht0 \hbox to
	0pt{\hss$#1\smile$\hss}\kern\wd0}
\newcommand{\ind}[1][]{\mathop{\mathpalette\@indepsymbol{}^{\!\!\!\!\rlap{$\scriptstyle\textnormal{#1}$}\,\,\,\,}}}
\newcommand{\nind}[1][]{\mathop{\mathpalette\@nindepsymbol{}^{\!\!\!\rlap{$\scriptstyle\textnormal{#1}$}\,\,\,}}}
\newcommand{\@Ind}[1][]{\mathpalette\@indepsymbol{}^{\!\!\!\!\mbox{$\scriptstyle\textnormal{#1}$}}}
\newcommand{\Ind}[1][]{\@Ind[\ \,]}
\newcommand{\newind}[4]{
	\newcommand{#1}{{\!\@Ind[#4]}}
	\newcommand{#2}{\ind[#4]}
	\newcommand{#3}{\nind[#4]}
}
\newind{\thInd}{\thind}{\nthind}{\thorn}
\renewcommand{\thInd}{\text{$\@Ind[\thorn]$\;}}
\makeatother

\newtheorem{thm}{Theorem}[section]

\newtheorem{lem}[thm]{Lemma}
\newtheorem{prop}[thm]{Proposition}

\theoremstyle{definition}

\theoremstyle{remark}

\theoremstyle{remark}

\theoremstyle{remark}

\newcommand {\Q} {\mathbb{Q}}

\renewcommand {\epsilon}{\varepsilon}

\title{On algebraic values of Weierstrass $\sigma$-functions}
\date{\today}
\author{Gareth Boxall}
\thanks{This work is based on the research supported in part by the National Research Foundation of South Africa (Grant Number 96234). The second author thanks the National Research Foundation of South Africa for funding for his doctoral studies. The third author is grateful to the Engineering and Physical Sciences Research Council for support under grant EP/N007956/1.}
\address{Gareth Boxall, Mathematics Division, Department of Mathematical Sciences, 
Stellenbosch University, Matieland 7602, South Africa; gboxall@sun.ac.za}
\author{Taboka Chalebgwa}
\address{Taboka Chalebgwa, Department of Mathematics and Statistics, McMaster University, Hamilton, Ontario, 8LS 4K1, Canada and Mathematics Division, Department of Mathematical Sciences, 
Stellenbosch University, Matieland 7602, South Africa; chalebgt@mcmaster.ca}
\author{Gareth Jones}
\address{Gareth Jones, School of Mathematics, University of Manchester, Oxford Road, Manchester, M13 9PL, UK; gareth.jones-3@manchester.ac.uk} 

\begin{document}
\maketitle

\begin{abstract}
Suppose that $\Omega$ is a lattice in the complex plane and let $\sigma$ be the corresponding Weierstrass $\sigma$-function. Assume that the point $\tau$ associated to $\Omega$ in the standard fundamental domain has imaginary part at most 1.9. Assuming that $\Omega$ has algebraic invariants $g_2,g_3$ we show that a bound of the form $c d^m (\log H)^n$ holds for the number of algebraic points of height at most $H$ and degree at most $d$ lying on the graph of $\sigma$. To prove this we apply results by Masser and Besson. What is perhaps surprising is that we are able to establish such a bound for the whole graph, rather than some restriction. We prove a similar result when, instead of $g_2,g_3$, the lattice points are algebraic. For this we naturally exclude those $(z,\sigma(z))$ for which $z\in\Omega$. 

\end{abstract}

\section{Introduction}

Let $\Omega$ be a lattice in $\mathbb{C}$. We are interested in the Weierstrass $\sigma$-function associated with $\Omega$. This is defined as $$\sigma_\Omega(z)=z\prod\limits_{\omega\in \Omega^*}\left(1-\dfrac{z}{\omega}\right)\exp\left(\dfrac{z}{\omega}+\dfrac{z^2}{2\omega^2}\right)$$ where $\Omega^*=\Omega\setminus\{0\}$. 

By definition, $\Omega=\omega_1\mathbb{Z}+\omega_2\mathbb{Z}$ where $\omega_1$ and $\omega_2$ are $\mathbb{R}$-linearly independent. It is known that $\omega_1$ and $\omega_2$ may be chosen so that $\frac{\omega_2}{\omega_1}$ lies in the upper half plane and has modulus at least 1 and real part in the interval $\left[-\frac{1}{2},\frac{1}{2}\right]$. Fix $\omega_1$ and $\omega_2$ with this property and set $\tau=\frac{\omega_2}{\omega_1}$. We prove two results concerning the number of algebraic points of bounded height and degree on the graph of $\sigma_\Omega$, under certain assumptions including that the imaginary part of $\tau$ is at most $1.9$. This bound on $\text{Im}(\tau)$ ensures that a useful growth condition is satisfied. Given an algebraic number $z$, we write $H(z)$ for its multiplicative height. For a pair $(z,w)$, $H(z,w)=\max\{H(z),H(w)\}$.

\begin{thm}\label{alglattice}
Suppose $\text{Im}(\tau)\leq 1.9$. If $\omega_1$ and $\omega_2$ are both algebraic then there is an effective positive constant $c$ such that, for all $d\geq e$ and $H\geq e^e$, there are at most $$c d^6(\log d)(\log H)^{2}\log\log H$$ algebraic points $(z,\sigma_\Omega(z))$ such that $[\mathbb{Q}(z,\sigma_\Omega(z)):\mathbb{Q}]\leq d$, $H(z,\sigma_\Omega(z))\leq H$ and $z\notin \Omega$.
\end{thm}

For our second result, we consider the parameters $$g_2=60\sum\limits_{\omega\in\Omega^*}\omega^{-4}\text{ and } g_3=140\sum\limits_{\omega\in\Omega^*}\omega^{-6}$$ associated with $\Omega$.

\begin{thm}\label{alggs}
Suppose $\text{Im}(\tau) \leq 1.9$. If $g_2$ and $g_3$ are both algebraic then there is an effective positive constant $c$ such that, for all $d\geq e$ and $H\geq e^e$, there are at most 
$$c d^{20}(\log d)^5 (\log H)^2 \log \log H$$
algebraic points $(z,\sigma_\Omega(z))$ such that $[\mathbb{Q}(z,\sigma_\Omega(z)):\mathbb{Q}]\leq d$ and $H(z,\sigma_\Omega(z))\leq H$. 
\end{thm}

There have been many recent results on the general topic of bounding the number of rational points of bounded height lying on the graph of some transcendental function. These go back to fundamental work of Bombieri and Pila, and then Pila \cite{BP,P}. In the latter of these papers, a bound of the form $cH^\epsilon$ is established for the number of rational points of height at most $H$ lying on the graph of a transcendental analytic function $f:[0,1]\to \mathbb{R}$. A construction by Surroca \cite{SurrocaShort} shows that this is best possible in the sense that there are examples for which no power of log bound will hold (see also \cite{PilaSurface}). Bounds of the form $c(\log H)^\gamma$ were given for functions satisfying pfaffian differential equations by Pila \cite{PilaPfaffCurve}, and for the Riemann zeta function, restricted to the interval $[2,3]$, by Masser \cite{Masser}. Pila's result is of a real nature, and wouldn't apply for instance to a pfaffian entire function if instead of rational points one considered points of bounded degree. Masser proves his result for points of bounded degree on a compact set containing $[2,3]$. 

Of the many other papers on this topic the most relevant to us are those by Besson \cite{Besson} and by the first and last authors \cite{BJ2}. In these, results similar to Theorems \ref{alglattice} and \ref{alggs} were proved, with no condition on $\tau$, but only for points lying on the graph of the restriction of $\sigma_\Omega$ to a disk. The result in \cite{Besson} is obtained with the aid of a zero estimate for $\sigma_\Omega$, also proved in \cite{Besson}, and this leads to a slightly better bound than is obtained \cite{BJ2}. The results in \cite{BJ2} apply generally to entire functions of finite order and positive lower order, with the exponent $\gamma$ in the bound $c(\log H)^\gamma$ depending only on the order and lower order. The constant $c$ depends the function, the disk and an upper bound $d$ for the degrees of the points considered. In the results in \cite{Masser} and \cite{Besson}, the dependence on $d$ was shown to be polynomial. For the interest in obtaining polynomial dependence on $d$ see \cite{HarryGaloisBounds,BJS}. 

What is striking about Theorem \ref{alggs} is that we do not have to restrict the function at all. We count algebraic points of bounded height and degree on the entire graph. For Theorem \ref{alglattice}, where the lattice points are algebraic, it is clear that the condition $z\notin\Omega$ could not be omitted. In earlier work of the authors, some results were obtained for restrictions of entire functions to unbounded sets, but generally a lot of points had to be omitted. In \cite{BJ1} the first and third authors obtained, for example, a bound for points on the graph of the Riemann zeta function restricted to a sector $-\frac{\pi}{2}<\theta\leq \arg{z}\leq \phi<\frac{\pi}{2}$. In recent work of the second author \cite{Chalebgwa}, bounds were given for restrictions of certain entire functions, defined by infinite products, to domains of the form $\mathbb{C}\setminus S$, where $S$ is an arbitrarily small sector whose interior contains the positive real line. 

To prove Theorems \ref{alglattice} and \ref{alggs}, we draw on the reasoning in \cite{Besson}. We combine Besson's zero estimate with a proposition of Masser from \cite{Masser}. Besson does this too, but we must work on increasingly large disks and our challenge is to show that, with perhaps a manageable number of exceptions, for any given $d$ and $H$ the $z$'s that we are counting all lie in a disk of sufficiently small radius. It is possible to avoid reliance on the zero estimate from \cite{Besson} and instead use techniques from earlier papers of the first and third authors, though at the cost of larger exponents. We expect it would also be possible to use the zero estimate of Coman and Poletsky, given as Corollary 7.2 in \cite{CP}. 

The current project began in the PhD research of the second author. A version of Theorem \ref{alglattice}, with $\Omega=\mathbb{Z}+\mathbb{Z}i$, was given in his thesis. 

In \S 2 we discuss a growth property of $\sigma_\Omega$ in the case where $\text{Im}(\tau)\le  1.9$. In \S 3 we prove Theorem \ref{alglattice}. In \S 4 we prove Theorem \ref{alggs}. Then in \S 5 we briefly outline a proof of a weaker version of Theorem \ref{alggs}, which avoids use of the zero estimate from \cite{Besson}.

\section{Growth}

Recall that $\Omega=\omega_1\mathbb{Z}+\omega_2\mathbb{Z}$, where $\tau=\frac{\omega_2}{\omega_1}$ lies in the upper half plane and has modulus at least 1 and real part in the interval $\left[-\frac{1}{2},\frac{1}{2}\right]$. Let $P$ be the closed region in $\mathbb{C}$ enclosed by the parallelogram with vertices $\frac{\pm\omega_1\pm \omega_2}{2}$. 

\begin{prop}\label{growth}
Suppose $\text{Im}(\tau)\leq 1.9$. There are effective positive constants $r$ and $c$ such that, for all $z\in\mathbb{C}$ with $|z|\geq r$, there exists $z_0\in P$ such that $z-z_0\in\Omega$ and $|\sigma_{\Omega}(z)|\geq |\sigma_{\Omega}(z_0)|e^{c|z|^2}$.
\end{prop}

The proof is based on reasoning given in \cite{Besson}.

\proof Let $\Omega'=\mathbb{Z}+\mathbb{Z}\tau$. Since $\Omega=\omega_1\Omega'$, we have $\sigma_{\Omega}(\omega_1 z)=\omega_1\sigma_{\Omega'}(z)$ and so it is sufficient to prove Proposition \ref{growth} with $\Omega'$, $\sigma_{\Omega'}$ and $P'=\frac{1}{\omega_1}P$ in place of $\Omega$, $\sigma_{\Omega}$ and $P$.

Let $z_0\in P'$. For any $n,m\in\mathbb{Z}$, $$\sigma_{\Omega'}(z_0+m+n\tau)=(-1)^{m+n+mn}\sigma_{\Omega'}(z_0)e^{(m\eta_1+n\eta_2)(z_0+\frac{m}{2} + \frac{n}{2}\tau)}$$ where $\eta_1$ and $\eta_2$ depend on $\tau$ (and we follow the convention which assigns them a value twice that in \cite{Besson}). The above equation is standard (see, for example, (20.21) on page 255 of \cite{MasserBook}). We are interested in how $$\text{Re}\left[\left(m\eta_1+n\eta_2\right)\left(z_0+\frac{m}{2} + \frac{n}{2}\tau\right)\right]$$ $$=\text{Re}\left(\frac{\eta_1}{2}\right)m^2+\text{Re}\left(\frac{\eta_1\tau+\eta_2}{2}\right)mn+\text{Re}\left(\frac{\eta_2\tau}{2}\right)n^2+\text{Re}\left(\eta_1z_0\right)m+\text{Re}\left(\eta_2z_0\right)n$$ behaves as $\max\{|m|,|n|\}$ increases. Following Besson, we consider the discriminant of $\text{Re}\left(\frac{\eta_1}{2}\right)m^2+\text{Re}\left(\frac{\eta_1\tau+\eta_2}{2}\right)mn+\text{Re}\left(\frac{\eta_2\tau}{2}\right)n^2$, $$\Delta=\left(\text{Re}\left(\frac{\eta_1\tau+\eta_2}{2}\right)\right)^2-4\left(\text{Re}\left(\frac{\eta_1}{2}\right)\right)\left(\text{Re}\left(\frac{\eta_2\tau}{2}\right)\right).$$ If $\Delta<0$ then $\left(\text{Re}\left(\frac{\eta_1}{2}\right)\right)\left(\text{Re}\left(\frac{\eta_2\tau}{2}\right)\right)\neq 0$ and $$\left|\text{Re}\left(\frac{\eta_1}{2}\right)m^2+\text{Re}\left(\frac{\eta_1\tau+\eta_2}{2}\right)mn+\text{Re}\left(\frac{\eta_2\tau}{2}\right)n^2\right|$$ $$\geq \max\left\{\left|\frac{\Delta}{4\text{Re}\left(\frac{\eta_2\tau}{2}\right)}\right|m^2,\left|\frac{\Delta}{4\text{Re}\left(\frac{\eta_1}{2}\right)}\right|n^2\right\},$$ by the general fact that if $a,b,c\in\mathbb{R}$ and $b^2-4ac<0$ then the minimum value of $\left|ax^2+bx+c\right|$ is $\left|\frac{4ac-b^2}{4a}\right|$. 

Since $\left|z_0\right|\leq 1+\left|\tau\right|$, $$\left|\text{Re}(\eta_1z_0)m+\text{Re}(\eta_2z_0)n\right|\leq (1+\left|\tau\right|)\max\{\left|\eta_1\right|,\left|\eta_2\right|\}\max\{\left|m\right|,\left|n\right|\}.$$ Let $c_1=\min\left\{\left|\dfrac{\Delta}{4\text{Re}\left(\frac{\eta_2\tau}{2}\right)}\right|,\left|\dfrac{\Delta}{4\text{Re}\left(\frac{\eta_1}{2}\right)}\right|\right\}$ and $c_2=(1+\left|\tau\right|)\max\{\left|\eta_1\right|,\left|\eta_2\right|\}$. Suppose $\Delta<0$. For each $z\in \mathbb{C}$, there exist $z_0\in P'$ and $m,n\in \mathbb{Z}$ such that $z=z_0+m+n\tau$. Then $$\text{Re}\left[(m\eta_1+n\eta_2)\left(z_0+\frac{m}{2} + \frac{n}{2}\tau\right)\right]\geq c_1\left(\max\{|m|,|n|\}\right)^2-c_2\max\{|m|,|n|\}.$$ Since $\max\{\left|m\right|,\left|n\right|\}\geq \dfrac{|z|-|z_0|}{1+\left|\tau\right|}$, we obtain positive constants $r$ and $c$, depending effectively on $\omega_1$ and $\omega_2$, such that $$\text{Re}\left[(m\eta_1+n\eta_2)\left(z_0+\frac{m}{2} + \frac{n}{2}\tau\right)\right]\geq c\left|z\right|^2$$ whenever $\left|z\right|\geq r$. Then, for all $z\in \mathbb{C}$ with $\left|z\right|\geq r$, there exists $z_0\in P'$ such that $z-z_0\in \Omega'$ and $$\left|\sigma_{\Omega'}(z)\right|\geq \left|\sigma_{\Omega'}(z_0)\right|e^{c|z|^2}.$$

It remains to show that $\Delta<0$ if $\text{Im}(\tau)\leq 1.9$. This is discussed in \cite{Besson}, but we provide further details. A simple calculation gives $$\Delta=\text{Im}(\tau)(\left|\eta_1\right|^2\text{Im}(\tau)-2\pi\text{Re}(\eta_1)).$$ This is done on page 5 of \cite{Besson}, but recall that the $\eta_1$ and $\eta_2$ there are half ours. We have $\eta_1=\frac{\pi^2}{3}E_2(\tau)$, where $$E_2(\tau)=1-24\sum\limits_{n\geq 1}\frac{nq^n}{1-q^n}$$ and $q=e^{2\pi i \tau}$ (see \cite{Lang}). Letting $y=\text{Im}(\tau)$ and $\phi(y)=\dfrac{24e^{-2\pi y}}{(1-e^{-2\pi y})^3}$, we have $$\text{Re}(E_2(\tau))\geq 1-\phi(y)$$ and $$\left|E_2(\tau)\right|\leq 1+\phi(y)$$ (see page 128 of \cite{IJT}). Therefore $\Delta<0$ provided $$y<\frac{6(1-\phi(y))}{\pi(1+\phi(y))^2}.$$ We compute a sequence $y_0,y_1,...$ with $y_0=\frac{\sqrt{3}}{2}$ and $$y_{n+1}=\frac{6(1-\phi(y_n))}{\pi(1+\phi(y_n))^2}.$$ We obtain $y_0<y_1<...<y_4\approx 1.909$.  Since $\phi(y)$ is positive and decreasing for $y\geq \frac{\sqrt{3}}{2}$, this ensures that $\Delta<0$ when $\frac{\sqrt{3}}{2}\leq y\leq 1.9$. The inequality $\frac{\sqrt{3}}{2}\leq y$ is implied by the fact that $\left|\tau\right|\geq 1$ and $-\frac{1}{2}\leq\text{Re}(\tau)\leq \frac{1}{2}$. \endproof

\section{When the lattice points are algebraic}

In this section we prove Theorem \ref{alglattice}. The following zero estimate was obtained in \cite{Besson}, for any lattice $\Omega=\omega_1\mathbb{Z}+\omega_2\mathbb{Z}$ (see Th\'eor\`eme 1.2 in \cite{Besson}).

\begin{thm}\label{Besson_zero} There is an effective positive constant $c$, depending only on $\omega_1$ and $\omega_2$, such that, for every integer $L\ge 1$, real number $R\ge 2$ and nonzero polynomial $P(X,Y)$, with complex coefficients and degree at most $L$ in each variable, the function $P(z,\sigma_\Omega(z))$ has at most $cL (R+\sqrt{L})^2 \log (R+L)$ zeroes in the disk $|z|\le R$.
\end{thm}

Besson uses this, in combination with a result of Masser, to prove the following (see Th\'eor\`eme 1.1 in \cite{Besson}).

\begin{thm}\label{Besson}
There is an effective positive constant $c$ such that, for all $d\geq 1$, $H\geq 3$ and $R\geq 2$, there are at most $$cR^{10}(\log R)\frac{d^4(\log H)^2}{\log(d\log H)}$$ algebraic points $(z,\sigma_\Omega(z))$ such that $[\mathbb{Q}(z,\sigma_\Omega(z)):\mathbb{Q}]\leq d$, $H(z,\sigma_\Omega(z))\leq H$ and $\left|z\right|\leq R$.
\end{thm}

This result and Theorem \ref{Besson_zero} are actually stated in \cite{Besson} under the assumption that $\omega_1=1$. However, it is clear from the argument that this is not needed. In particular, the identity $\sigma_{\Omega}(\omega_1 z)=\omega_1\sigma_{\Omega'}(z)$, when $\Omega'=\mathbb{Z}+\mathbb{Z}\tau$ and $\Omega=\omega_1\Omega'$, ensures that Theorem \ref{Besson_zero} holds for any lattice. 

In order to apply Besson's work to prove Theorem \ref{alglattice}, we show that, for each $d$ and $H$, all the relevant $z$'s lie in a disk whose radius is polynomial in $d$ and $\log H$. For values of $z$ not close to a lattice point, we use Proposition \ref{growth}. We do not consider the lattice points themselves. That leaves only those values of $z$ which are close but not equal to lattice points. We use basic facts about height and degree to eliminate the possibility of any of these. 

We could then simply apply Theorem \ref{Besson} to obtain a result along the lines of Theorem \ref{alglattice}. However, this would lead to a weaker estimate. So instead we use Besson's method of proof. This involves combining his zero estimate with Proposition 2 of \cite{Masser}, which is as follows.

\begin{prop}\label{propn2} Given integers $d\ge 1, T\ge \sqrt{8d}$ and positive real numbers $A,Z,M$ and $H$, with $H\geq 1$, let $f_1$ and $f_2$ be analytic functions on a neighbourhood of the disk $|z|\le 2Z$. Suppose $|f_1(z)|\le M$ and $|f_2(z)|\le M$ for all $\left|z\right|\leq 2Z$. Let $\mathcal Z\subseteq\mathbb{C}$ be finite such that, for all $z,z'\in \mathcal Z$, 
\begin{itemize}
\item[(a)] $|z|\le Z$,
\item[(b)] $ |z-z'| \le 1/A$,
\item[(c)] $[\mathbb Q (f_1(z),f_2(z)) :\mathbb{Q}]\le d$ and
\item[(d)] $H(f_1(z),f_2(z)) \le H$.
\end{itemize}
If 
\begin{equation}\label{propn2condition}
(AZ)^T > (4T)^{96d^2/T} (M+1)^{16 d} H^{48d^2}
\end{equation}
then there is a non-zero polynomial $P \in \mathbb{C}[X,Y]$, of total degree at most $T$, such that $$P(f_1(z),f_2(z))=0$$ for all $z \in \mathcal Z$.

\end{prop}

We now start working towards our bound on the radius. As in the previous section, let $P$ be the parallelogram with vertices $\frac{\pm \omega_1 \pm \omega_2}{2}$. Since $\lim\limits_{w\rightarrow 0}\frac{\sigma_\Omega(w)}{w}=1$ and $w=0$ is the only zero of $\sigma_\Omega$ in $P$, there is some $\delta>0$ such that $\left|\log\left|\sigma_\Omega(w)\right|-\log\left|w\right|\right|\leq 1$ whenever $\log|\sigma_\Omega(w)|\leq -\delta$. We fix some such $\delta$. Recall that $\tau=\frac{\omega_2}{\omega_1}$. 

\begin{lem}\label{common}
Assume $\text{Im}(\tau)\leq 1.9$. Let $d\geq 1$ and $H\geq e$. Suppose $z\in \mathbb{C}\setminus\Omega$ is such that $\sigma_{\Omega}(z)$ is algebraic with $H(\sigma_\Omega(z))\leq H$ and $[\mathbb{Q}(\sigma_{\Omega}(z)):\mathbb{Q}]\leq d$. Let $z_0\in P$ be such that $z-z_0\in \Omega$. Assume $\left|z\right|\geq r$, where $r$ is as in Proposition \ref{growth}. For all $B>0$ there exists $A>0$, depending only on $B, \omega_1$ and $\omega_2$, such that, for all $N\geq \sqrt{d\log H}$, if $\left|z\right|\geq AN$ then $\log\left|z_0\right|\leq -BN^2$. 
\end{lem}

\proof Let $B>0$, $A>0$ and $N\geq \sqrt{d\log H}$. Let $c$ be as in Proposition \ref{growth}. Then $$|\sigma_\Omega(z)|\geq |\sigma_\Omega(z_0)|e^{c|z|^2}.$$ Since $\sigma_\Omega(z)$ is algebraic with degree at most $d$ and height at most $H$, $|\sigma_\Omega(z)|\leq H^d$. So $$N^2\geq d\log H\geq\log\left|\sigma_\Omega(z)\right|\geq \log\left|\sigma_\Omega(z_0)\right| + c\left|z\right|^2.$$ If $\left|z\right|\geq AN$ then $$N^2\geq \log\left|\sigma_\Omega(z_0)\right| + cA^2N^2$$ and so $$\log\left|\sigma_\Omega(z_0)\right|\leq (1-cA^2)N^2$$ which implies $$\log\left|z_0\right|\leq (1-cA^2)N^2+1\leq (2-cA^2)N^2$$ if $A$ has been chosen so that $$1-cA^2\leq -\delta.$$ We obtain the result with $A=\max\left\{\sqrt{\frac{1+\delta}{c}},\sqrt{\frac{2+B}{c}}\right\}$. \endproof

The following lemma must be well known and is easily proved using the cosine rule. 

\begin{lem}\label{cosine}
There is a positive constant $c$, depending only on $\omega_1$ and $\omega_2$, such that for all $\omega\in \Omega$ there are integers $k, l$ such that $\left|k\right|\leq c\left|\omega\right|$, $\left|l\right|\leq c\left|\omega\right|$ and $\omega=k\omega_1+l\omega_2$. 
\end{lem}

\begin{prop}\label{applgrowth}
Assume $\text{Im}(\tau)\leq 1.9$ and both $\omega_1$ and $\omega_2$ are algebraic. Let $d\geq 1$ and $H\geq e$. Suppose $z\in \mathbb{C}\setminus\Omega$ is such that both $z$ and $\sigma_{\Omega}(z)$ are algebraic with $H(z,\sigma_\Omega(z))\leq H$ and $[\mathbb{Q}(z,\sigma_{\Omega}(z)):\mathbb{Q}]\leq d$. There exists $A>0$, depending only on $\omega_1$ and $\omega_2$, such that $|z|\leq A d\sqrt{\log H}$. 
\end{prop}

\proof The constants $c_1,c_2,c_3$ in this proof will only depend on $\omega_1$ and $\omega_2$. Choose $z_0\in P$ such that $z-z_0\in\Omega$ and let $\omega=z-z_0$. Then $\left|\omega\right|\leq \left|z\right|+\left|\omega_1\right|+\left|\omega_2\right|$. Combining this with Lemma \ref{cosine}, we obtain $c_1>0$ and integers $k$ and $l$ such that $\omega=k\omega_1+l\omega_2$ and $\left|k\right|,\left|l\right|\leq c_1\left|z\right|$.  

Since $z_0=z-\omega$, we have $$H(z_0)\leq 2H(z)H(\omega)\leq 4H(z)H(k)H(\omega_1)H(l)H(\omega_2).$$ Since $H(z)\leq H$, $[\mathbb{Q}(z):\mathbb{Q}]\leq d$, $H(k)=\left|k\right|$ and $H(l)=\left|l\right|$, we have $\left|z\right|\leq H^d$ and then $H(k),H(l)\leq c_1H^d$. So there exists $c_2>0$ such that $$\log (H(z_0))\leq (2d+1)\log H+c_2.$$ The degree of $\omega$ is at most $[\mathbb{Q}(\omega_1,\omega_2):\mathbb{Q}]$ and so $$[\mathbb{Q}(z_0):\mathbb{Q}]=[\mathbb{Q}(z-\omega):\mathbb{Q}]\leq [\mathbb{Q}(\omega_1,\omega_2):\mathbb{Q}]d.$$ Letting $c_3=[\mathbb{Q}(\omega_1,\omega_2):\mathbb{Q}]$, it follows that $$\log\left|z_0\right|\geq -[\mathbb{Q}(z_0):\mathbb{Q}]\log(H(z_0))\geq -c_3d((2d+1)\log H+c_2)> -3c_3(1+c_2)d^2\log H.$$ Letting $B=3c_3(1+c_2)$ and $N=d\sqrt{\log H}$, the result follows by Lemma \ref{common} (assuming, as we may, that $\left|z\right|$ is at least the $r$ of Proposition \ref{growth}). \endproof

We can now complete the proof of Theorem \ref{alglattice}. We apply Proposition \ref{propn2} with $f_1(z)=z$ and $f_2(z)=\sigma_\Omega(z)$. We set $Z= 4 A d\sqrt{\log H}$, with the $A$ from Proposition \ref{applgrowth}, and put
\[
\mathcal Z = \{ z \in \mathbb C\setminus\Omega : [\mathbb Q(z,\sigma_\Omega(z)) : \mathbb Q]\le d , H(z,\sigma_\Omega(z) ) \le H\}.
\] 
By Proposition \ref{applgrowth}, we have $|z|\le Z$ for all $z\in \mathcal Z$. We take $A$ in Proposition \ref{propn2} to be $2/Z$. From the identity $$\sigma_{\Omega}(z_0+m\omega_1+n\omega_2)=(-1)^{m+n+mn}\sigma_{\Omega}(z_0)e^{(m\eta_1+n\eta_2)(z_0+\frac{m}{2}\omega_1 + \frac{n}{2}\omega_2)}$$ (see (20.21) on page 255 of \cite{MasserBook}) and Lemma \ref{cosine}, there exist positive constants $c_1$ and $c_2$ such that $\left|\sigma_\Omega(z)\right|\leq c_1e^{c_2\left|z\right|^2}$ for all $z\in\mathbb{C}$. We set $M = c_1e^{c_2 Z^2}$. With these choices, (a) - (d) of Proposition \ref{propn2} hold. If we take $T=c_3 d^3 \log H$, for a suitable $c_3>0$, then we will have \eqref{propn2condition}. We then apply Besson's zero estimate, Theorem \ref{Besson_zero}, to conclude.

\section{When $g_2$ and $g_3$ are algebraic}

In this section we prove Theorem \ref{alggs}. Our approach is similar to that followed in \S 3 and we again use Theorem \ref{Besson_zero}. We suppose throughout that $g_2$ and $g_3$ are algebraic. First, we require a transcendence measure, a very special case of a theorem of David and Hirata-Kohno.
\begin{lem}\label{measure}There is a positive $c$, effectively computable from $g_2$ and $g_3$, with the following property. Let $H\ge 3, d\ge 1$ and suppose $\alpha$ is algebraic, with $H(\alpha)\le H$ and $[\Q(\alpha):\Q]\le d$. If $\omega\in\Omega$ is nonzero then 
\[
 \log |\alpha-\omega| \ge -c d^4 (\log d)^2 (\log H) |\omega|^2 (1+\max\{0,\log |\omega|\})^3.
 \]
 \end{lem}
 \begin{proof}  This is immediate from Theorem 1.6 in \cite{DavidHirataKohno}.
\end{proof}

We now obtain a result similar to Proposition \ref{applgrowth}. This time we allow for some exceptions but can ensure that the set of exceptional $z\in\mathbb{C}$ has $\mathbb{Q}$-linear dimension at most one.

\begin{prop}\label{applgrowth2}
Assume $\text{Im}(\tau)\leq 1.9$. There are positive constants $A$ and $B$ such that, for all $d,H\geq e$ and $z,z'\in \mathbb{C}$, if $z,z',\sigma_\Omega(z)$ and $\sigma_\Omega(z')$ are algebraic, with $H(z,\sigma_\Omega(Z)), H(z',\sigma_\Omega(z'))\leq H$ and $[\mathbb{Q}(z,\sigma_\Omega(z)):\mathbb{Q}], [\mathbb{Q}(z',\sigma_\Omega(z')):\mathbb{Q}]\leq d$, then $$\min\{|z|,|z'|\}\leq A\sqrt{d^9(\log d)^2\log H}$$ or there exist $\omega,\omega'\in\Omega$ such that $$\max\{\log|z-\omega|,\log|z'-\omega'|\}<-Bd^9(\log d)^2\log H\text{ and }\frac{z'}{z}=\frac{\omega'}{\omega}\in\mathbb{Q}.$$
\end{prop}

\proof Let $d, H\geq e$ and let $z,z'\in\mathbb{C}$ be such that $H(z,\sigma_\Omega(Z)), H(z',\sigma_\Omega(z'))\leq H$ and $[\mathbb{Q}(z,\sigma_\Omega(z)):\mathbb{Q}], [\mathbb{Q}(z',\sigma_\Omega(z')):\mathbb{Q}]\leq d$. Choose $z_0, z_0'\in P$ so that $z-z_0,z'-z_0'\in\Omega$. With $r$ as in Proposition \ref{growth}, we may assume $\left|z\right|,\left|z'\right|\geq r$. Applying Lemma \ref{common}, with any $B>0$ and with $N=\sqrt{d^9(\log d)^2\log H}$, we obtain $A>0$ such that $$\min\{\left|z\right|,\left|z'\right|\}\leq A\sqrt{d^9(\log d)^2\log H}$$ or $$\max\{\log\left|z_0\right|,\log\left|z_0'\right|\}\leq-Bd^9(\log d)^2\log H.$$ At some points in the argument we shall want $B$ to be larger than certain values computable from $g_1$ and $g_2$, and may assume that this is the case.

If $\min\{|z|,|z'|\}> A\sqrt{d^9(\log d)^2\log H}$ then we have $\omega,\omega'\in\Omega$ such that $$\log|z-\omega|=\log|z_0|\leq-Bd^9(\log d)^2\log H$$ and $$\log|z'-\omega'|=\log\left|z_0\right|\leq-Bd^9(\log d)^2\log H.$$ We may assume $\omega\omega'\neq 0$. We show that $\frac{\omega'}{\omega}\in\mathbb{Q}$. Suppose not. Then $\frac{\omega'}{\omega}\notin\mathbb{R}$ and we have some $a,b\in \mathbb{Q}$ such that $a\omega+b\omega'=\omega_1$. We want to bound the height of $a$ and $b$ and for this we consider how they are obtained. By Lemma \ref{cosine}, we have $k\omega_1+l\omega_2=\omega$ and $k'\omega_1+l'\omega_2=\omega'$ for some integers $k,l,k',l'$ such that such that $\left|k\right|,\left|l\right|\leq c\left|\omega\right|$ and $\left|k'\right|,\left|l'\right|\leq c\left|\omega'\right|$, where $c$ is positive and depends only on $g_1$ and $g_2$. We may assume $c\geq 1$. Since $\frac{\omega'}{\omega}\notin\mathbb{R}$, $kl'-lk'\neq 0$.

We then have $a=\frac{l'}{kl'-lk'}$ and $b=\frac{-l}{kl'-lk'}$. It follows that $$H(a,b)\leq 2c^2\max\{1,\left|\omega\right|\}\max\{1,\left|\omega'\right|\}.$$ Since $z$ and $z'$ have height at most $H$ and degree at most $d$, we have $\left|z\right|,\left|z'\right|\leq H^d$. We also have $\left|z-\omega\right|<1$ and $\left|z'-\omega'\right|<1$ and so $\left|\omega\right|\leq \left|z\right|+1$ and $\left|\omega'\right|\leq \left|z'\right|+1$. Then $$H(a,b)\leq 2c^2(H^d+1)^2\leq 8c^2H^{2d}$$ and $$H(\left|a\right|+\left|b\right|)\leq 16c^2H^{2d}.$$

Now $$\log|(az+bz')-\omega_1|=\log\left|a(z-\omega)+b(z'-\omega')\right|$$ $$\leq\log(|a|+|b|)-Bd^9(\log d)^2\log H$$ $$\leq \log(16c^2)+2d\log H-Bd^9(\log d)^2\log H.$$ 

We have $H(az+bz')\leq 2(8c^2)^2H^{4d+2}$ and $[\mathbb{Q}(az+bz'):\mathbb{Q}]\leq d^2$. By Lemma \ref{measure} there is a positive constant $c'$, depending effectively on $g_1$ and $g_2$, such that $$\log |(az+bz')-\omega_1|\ge -c' d^9(\log d)^2 \log H |\omega_1|^2 (1+\max\{0,\log |\omega_1|\})^3.$$ We have a contradiction given that $B$ is large enough. So $\frac{\omega'}{\omega}\in\mathbb{Q}$. 

We complete the proof by showing that $\frac{z'}{z}=\frac{\omega'}{\omega}$. Since $\omega=k\omega_1+l\omega_2$, $\omega'=k'\omega_1+l'\omega_2$ and $\omega_1$ and $\omega_2$ are $\mathbb{Q}$-linearly independent, the rationality of $\frac{\omega'}{\omega}$ implies that there exists $q\in \mathbb{Q}$ such that $qk=k'$ and $ql=l'$. In fact, $q=\frac{\omega'}{\omega}$. We then have $$H\left(\frac{\omega'}{\omega}\right)\leq\max\{\left|k\right|,\left|k'\right|\}\leq c(\max\{\left|z\right|,\left|z'\right|\}+1)\leq c(H^d+1).$$

Assuming without loss of generality that $|\omega|\geq|\omega'|$, note that $$\left|\frac{\omega' z}{\omega}-z'\right|= \left|\frac{\omega' z}{\omega}-\omega'+\omega'-z'\right|\leq\frac{\left|\omega'\right|}{\left|\omega\right|}\left|z-\omega\right|+\left|\omega'-z'\right|\leq \left|z-\omega\right|+\left|z'-\omega'\right|$$ and so $$\log\left|\frac{\omega' z}{\omega}-z'\right|\leq \log{2}-Bd^9(\log d)^2\log H.$$ But $\frac{\omega' z}{\omega}-z'$ is algebraic with $$H\left(\frac{\omega' z}{\omega}-z'\right)\leq 2H\left(\frac{\omega'}{\omega}\right)H(z)H(z')\leq 2c(H^d+1)H^2$$ and $$\left[\mathbb{Q}\left(\frac{\omega' z}{\omega}-z'\right):\mathbb{Q}\right]\leq d^2.$$ Given that $B$ is large enough, we must have $\frac{\omega'z}{\omega}-z'=0$ and so $\frac{z'}{z}=\frac{\omega'}{\omega}$. \endproof

Proposition \ref{applgrowth2} is almost enough for our purposes. To finish the proof of Theorem \ref{alggs}, we bound the cardinality of the exceptional set
\[
\mathcal{E}=\{ z \in \mathbb{C} : H(z,\sigma_\Omega(z) )\leq H, [\mathbb Q(z,\sigma_\Omega(z)) :\mathbb Q] \le d, |z| > \max\{r,A \sqrt{ d^9 (\log d)^2 \log H}\}\}
\]
where $d,H\geq e$, $A$ is as in Proposition \ref{applgrowth2} and $r$ is as in Proposition \ref{growth}. Suppose that $z,z'\in\mathcal{E}$ and let $\omega,\omega'\in \Omega$ be such that $\max\{\log|z-\omega|,\log|z'-\omega'|\}<-Bd^9(\log d)^2\log H$ and $\frac{z'}{z}=\frac{\omega'}{\omega}\in\mathbb{Q}$, with $B$ as in Proposition \ref{applgrowth2}. We may assume $\omega'\neq 0$. Let $\omega^*\in\Omega$ be non-zero and of minimum modulus lying on the line through $0$ and $\omega$. So $\omega = p \omega^*$ for some non-zero integer $p$. Let $z^*=z/p$. Then $z'$ is also an integer multiple of $z^*$, and so it suffices to show that there are few integers $p$ for which $pz^*\in\mathcal{E}$. We shall use $c_1,...,c_{11}$ to denote various positive constants which only depend on $g_1$ and $g_2$. 

We can write $\omega= k \omega_1+l\omega_2$,  with $k,l\in \mathbb{Z}$. By Lemma \ref{cosine} we have $\left|k\right|,\left|l\right|\leq c_1\left|\omega\right|$. Since $p$ divides both $k$ and $l$, we have $|p| \le c_1|\omega| \le c_1 (|z|+1) \le c_2 H^d$. And so  $\log H(z^*)\leq c_3 d\log H$. 

By Proposition \ref{growth}, $$\log\left|\sigma_\Omega(z)\right|\geq \log\left|\sigma_\Omega(z_0)\right|+c_4|z|^2\geq \log\left|z_0\right|-c_5+c_4|z|^2$$ $$=\log\left|z-\omega\right|-c_5+c_4|z|^2\geq \log\left|z^*-\omega^*\right|-c_5+c_4p^2|z^*|^2$$ where $z_0\in P$ is such that $z-z_0\in\Omega$. By Lemma \ref{measure}, $$\log\left|z^*-\omega^*\right|\geq-c_6d^4(\log d)^2d(\log H)\left|\omega^*\right|^2(1+\max\{0,\log\left|\omega^*\right|\})^3$$ $$\geq-c_7d^5(\log d)^2(\log H)\left|z^*\right|^2(1+\max\{0,\log\left|z^*\right|\})^3.$$

We then have $$d\log H\geq \log\left|\sigma_\Omega(z)\right|\geq -c_7d^5(\log d)^2(\log H)\left|z^*\right|^2(1+\max\{0,\log\left|z^*\right|\})^3-c_5+c_4p^2\left|z^*\right|^2.$$

It follows that $$p^2\leq c_8d^5(\log d)^2(\log H)(1+\max\{0,\log|z^*|\})^3\leq  c_9d^5(\log d)^2(\log H)(1+d\log H)^3,$$ and $\mathcal{E}$ has size at most 
$$
c_{10}\sqrt{d^5(\log d)^2(\log H)(1+d\log H)^3}.
$$

So when proving Theorem \ref{alggs} it suffices to consider those $z$ with $|z|\le A\sqrt{d^9(\log d)^2\log H}$. We then follow the proof of Theorem \ref{alglattice}, with this larger radius. Set $Z= 4A \sqrt{d^9(\log d)^2\log H}$. Making appropriate modifications in the choice of parameters for Proposition \ref{propn2}, there is a non-zero polynomial $P$ of degree at most $c_{11} d^{10}(\log d)^2 \log H$ such that it suffices to bound the number of zeroes of $P(z,\sigma_\Omega(z))$ with $|z| \le  A\sqrt{d^9(\log d)^2\log H}$. We apply Theorem \ref{Besson_zero} to conclude.

\section{Alternative reasoning}

We outline an alternative argument that uses reasoning from \cite{BJ2}, instead of Theorem \ref{Besson_zero}, to obtain a conclusion similar to that of Theorem \ref{alggs}, but with a slightly weaker bound. As in the previous section, we obtain constants $c_{10}, c_{11}$ and $A$ and a polynomial $P$, with degree at most $T=c_{11} d^{10}(\log d)^2 \log H$, such that, with at most $c_{10}\sqrt{d^5(\log d)^2\log H(1+d\log H)^3}$ exceptions, all $z\in \mathbb{C}$ with $\left[\mathbb{Q}(z,\sigma_\Omega(z)):\mathbb{Q}\right]\leq d$ and $H(z,\sigma_\Omega(z))\leq H$ lie in the closed disk with centre zero and radius $A\sqrt{d^9(\log d)^2\log H}$, and satisfy the equation $P(z,\sigma_\Omega(z))=0$. In the proof of Proposition 2 of \cite{Masser}, Masser provides further information about the coefficients of $P$ assuming, as we may, that there are enough points $(z,\sigma_\Omega(z))$ such that $H(z,\sigma_\Omega(z))\leq H$, $[\mathbb{Q}(z,\sigma_\Omega(z)):\mathbb{Q}]\leq d$ and $\left|z\right|\leq A\sqrt{d^9(\log d)^2\log H}$. Specifically, $P$ may be chosen to have integer coefficients, each with absolute value at most $2^{1/d}(T+1)^2H^T$.

We use reasoning from the proof of Theorem 2.5 in \cite{BJ2} to bound the number of zeroes of $P(z,\sigma_\Omega(z))$ in the closed disk with centre zero and radius $A\sqrt{d^9(\log d)^2\log H}$. Let $k$ be the degree of $Y$ in $P(X,Y)$. We get a good bound if $k=0$. Assume $k\geq 1$. We set $\tilde{P}(X,Y)=Y^kP(X,\frac{1}{Y})$, $R(X)=\tilde{P}(X,0)$, $Q(X,Y)=\tilde{P}(X,Y)-R(X)$ and $\tilde{Q}(X,Y)=\frac{1}{Y}Q(X,Y)$. These are all polynomials with integer coefficients and we have bounds for their degrees and the absolute values of their coefficients. From now on we shall use $c$ to represent various positive constants, possibly different on each occasion, depending effectively on $g_2$ and $g_3$. 

We seek some $w\in\mathbb{C}$ such that $\left|w\right|$ is not too large and $\left|P(w,\sigma_\Omega(w))\right|$ is not too small, restricting our attention to those $w$ for which $\left|\sigma_\Omega(w)\right|=M_{\sigma_\Omega}(\left|w\right|)$, where $M_{\sigma_\Omega}(s)$ denotes $\max\{\left|\sigma_\Omega(z)\right|:\left|z\right|\leq s\}$. We first find a disk around zero outside of which, for any such $w$, $\left|Q\left(w,\frac{1}{\sigma_\Omega(w)}\right)\right|\leq \frac{1}{2}$. A disk with radius $cT$ will work as, outside of that, we shall have $\left|w\right|>cT$, $$\left|\tilde{Q}\left(w,\frac{1}{\sigma_\Omega(w)}\right)\right|\leq 2^{1/d}(T+1)^4H^T\left|w\right|^T\max\{1,\left|\sigma_\Omega(w)\right|^{-T}\},$$ $$\left|\sigma_\Omega(w)\right|=M_{\sigma_\Omega}(\left|w\right|)\geq ce^{c\left|w\right|^2}$$ and $$\left|Q\left(w,\frac{1}{\sigma_\Omega(w)}\right)\right|=\left|\sigma_\Omega(w)\right|^{-1}\left|\tilde{Q}\left(w,\frac{1}{\sigma_\Omega(w)}\right)\right|.$$ 

We then choose a $w$ outside but close to this disk, with $\left|\sigma_\Omega(w)\right|=M_{\sigma_\Omega}(\left|\omega\right|)$ and $\left|R(w)\right|\geq 1$. By the Boutroux-Cartan lemma (stated as Fact 2.3 in \cite{BJ2}), there is such a $w$ in the disk with centre zero and radius $cT+T+14.$ We then have $$\left|\tilde{P}\left(w,\frac{1}{\sigma_\Omega(w)}\right)\right|\geq \frac{1}{2}$$ and so $$\left|P(w,\sigma_\Omega(w))\right|=\left|\sigma_\Omega(w)\right|^k\left|\tilde{P}\left(w,\frac{1}{\sigma_\Omega(w)}\right)\right|\geq \frac{1}{2}.$$ We are concerned with points on the disk $D_1=\{z\in\mathbb{C}:\left|z\right|\leq A\sqrt{d^9(\log d)^2\log H}\}$. We work with the disk $D_2=\{z\in\mathbb{C}:\left|z-w\right|\leq cT\}$, where this $c$ is chosen large enough to ensure that $2D_1\subseteq D_2$. We have $\max\{\left|P(z,\sigma_\Omega(z))\right|:z\in D_2\}\leq (T+1)^4H^T(cT)^Te^{cT^3}$.

By a well-known consequence of Jensen's formula (stated as Fact 2.4 in \cite{BJ2}), $P(z,\sigma_\Omega(z))$ has at most $c\log\left((T+1)^4H^T(cT)^Te^{cT^3}\right)$ zeroes in $D_2$ and hence also in $D_1$. So we obtain the conclusion of Theorem \ref{alggs} but with the weaker bound of $cd^{30}(\log d)^6(\log H)^3$.

\bibliographystyle{amsplain}
\bibliography{refs}

\end{document}